\documentclass{elsarticle}

\usepackage{lineno,hyperref}
\modulolinenumbers[1]

\usepackage{graphicx}
\usepackage{amssymb}
\usepackage{mathtools}
\usepackage{amsthm}
\usepackage{dsfont}
\usepackage{epsfig}
\usepackage{float}
\usepackage{epstopdf}
\usepackage{xcolor}
\usepackage[margin=0.9in]{geometry}

\newtheorem{theorem}{Theorem}[section]

\newtheorem{corollary}[theorem]{Corollary}
\newtheorem{definition}[theorem]{Definition}

\journal{Computational and Applied Mathematics}

\begin{document}
	
	\begin{frontmatter}
		
		\title{Quantum Trigonometric B\'ezier Curves}

		\author{\c{C}etin Di\c{s}ib\"{u}y\"{u}k}
		\ead{cetin.disibuyuk@deu.edu.tr}

		\address{Department of Mathematics, Dokuz Eyl\"ul University, Fen Fak\"ultesi,T{\i}naztepe Kamp\"us\"u, 35390  Buca, \.Izmir.}
		
		\begin{abstract}
			
			In order to construct quantum trigonometric B\'ezier curves with shape parameter, one parameter family of trigonometric Bernstein basis functions are introduced. We study the total positivity of the basis functions to analyze the shape preserving properties of the quantum trigonometric B\'ezier curves. We also showed that quantum trigonometric B\'ezier curves can be evaluated by two different recursive evaluation  algorithms. Finally, we have defined rational counterpart of quantum trigonometric B\'ezier curves and show that the rational quantum trigonometric B\'ezier curves posses nice shape preserving properties.
			
		\end{abstract}
		
	\end{frontmatter}
	
	\section{Introduction}\label{intro}
	
	B\'ezier curves are defined for various types of spaces such as polynomials, trigonometric polynomials, hyperbolic polynomials and special M\"untz spaces \cite{alfeld,disibuyuk15,li}. Although B\'ezier curves have simple definition and good shape properties, they have some limitations because the shapes of B\'ezier curves are only defined by their control points \cite{hu18}. To overcome these limitations researchers generalizes B\'ezier curves by developing B\'ezier curves with shape parameters, see \cite{disibuyuk07,disibuyuk08,hu18,lewanowicz04,li} and references therein. These shape parameters allow us to change the shape of the B\'ezier curve without changing its control points.

		This paper is about generalizing the trigonometric B\'ezier curves defined in \cite{alfeld} which discuss a theory of Bernstein-B\'ezier polynomials that has close connections to trigonometric polynomials. Trigonometric B\'ezier curves over the interval $ [a,b] $ is defined as
		\begin{equation}\label{trigbezcurve}
			P(x)=\sum_{k=0}^{n}\mathbf{b}_{k}B^n_k(x),
		\end{equation}
		where the trigonometric Bernstein basis functions $ B^n_k (x)$ are
		\begin{align}\label{trigbases}
			\binom{n}{k}\left(\frac{\sin(x-a)}{\sin(b-a)}\right)^{k}\left(\frac{\sin(b-x)}{\sin(b-a)}\right)^{n-k},\quad k=0,1,\ldots ,n.
		\end{align}
		The basis functions \eqref{trigbases} are developed by using circular barycentric coordinates given in \cite{gonsorneamtu}. Trigonometric Bernstein bases are also a special case of the bases defined in \cite{disibuyuk15}.
		
		Note that trigonometric B\'ezier curves are an element of the space $ \pi_n = \operatorname{span}\left\{\sin^{k}(x)\cos^{n-k}\right\}_{k=0}^{n}. $ Although the trigonometric B\'ezier curves and classical B\'ezier curves belong to different spaces, they share several basic properties, such as de Casteljau algorithm and subdivision (see, \cite{alfeld,disibuyuk15}). The curves that share shape properties with classical B\'ezier curves are called Bernstein-B\'ezier like curves. The curves studied in \cite{alfeld,disibuyuk07,disibuyuk08,disibuyuk15,hu18,lewanowicz04,li} are all share common shape preserving properties with classical B\'ezier curves.
	%
	%
	%These trigonometric Bernstein bases over the interval $ [a,b] $ are defined as 
	%\begin{align}\label{trigbases}
	%	\binom{n}{k}\left(\frac{\sin(x-a)}{\sin(b-a)}\right)^{k}\left(\frac{\sin(b-x)}{\sin(b-a)}\right)^{n-k},\quad k=0,1,\ldots ,n.
	%\end{align}
	%The basis functions \eqref{trigbases} are developed by using circular barycentric coordinates given in \cite{gonsorneamtu}. Trigonometric Bernstein bases are also a special case of the bases defined in \cite{disibuyuk15}.
	%
	%{\color{red}\vspace{\parindent} Using the basis \eqref{trigbases}, trigonometric B\'ezier curves defined in \cite{alfeld} as
		%\begin{equation}\label{trigbezcurve}
		%	P(x)=\sum_{k=0}^{n}\mathbf{b}_{k}\binom{n}{k}\left(\frac{\sin(x-a)}{\sin(b-a)}\right)^{k}\left(\frac{\sin(b-x)}{\sin(b-a)}\right)^{n-k}
		%\end{equation}
		%where $ \mathbf{b}_{k},\ k=,0,1,\ldots,n $ are the control points of the curve $ P(x). $ Trigonometric B\'ezier curves are an element of the space $ \pi_n = \operatorname{span}\left\{\sin^{k}(x)\cos^{n-k}\right\}_{k=0}^{n}. $ Although the trigonometric B\'ezier curves and classical B\'ezier curves belong to different spaces, they share several basic properties, such as de Casteljau algorithm and subdivision (see, \cite{alfeld,disibuyuk15}).
		%}
	
	There are many Bernstein-B\'ezier like curves in the literature. The most astonishing ones among these generalizations are $ q $-B\'ezier curves which are introduced a few decades ago (see \cite{oruc99},\cite{oruc03}). $ q $-B\'ezier curves are also called quantum B\'ezier curves which is one of the popular research areas in CAGD.
	
	\textit{Our main motivation is to extend quantum theory to trigonometric B\'ezier curves}. This extension allow us to define B\'ezier curves with shape parameter on the space $ \pi_n. $ We proceed in the following fashion. In Section \ref{secbases}, we define quantum trigonometric Bernstein basis functions. Total positivity of quantum trigonometric Bernstein basis functions, which is important for analyzing the shape properties of curves, is discussed in subsection \ref{sectotalpositivity}. The quantum trigonometric B\'ezier curves are defined in Section \ref{subsecquanbezcurve} where we also give two different de Casteljau type evaluation algorithm for these curves.
	%In Section \ref{sectotalpositivity} shape preserving properties of quantum trigonometric B\'ezier curves discussed by investigating total positivity of quantum trigonometric Bernstein basis functions.
	Finally in Section \ref{secrational}, we define rational counterpart of quantum trigonometric B\'ezier curves and show that these rational curves posses nice shape preserving properties.
	
	\section{Quantum trigonometric Bernstein bases}\label{secbases}
	
	For simplicity, we will use the following notation:
	\[
	d(x,y;q):=\frac{q+1}{2}\sin(y-x)+\frac{q-1}{2}\sin(y+x).
	\]
	Note that the ratios
	\[
	\frac{d(a,x;1)}{d(a,b;1)}
	\]
	and
	\[
	\frac{d(x,b;1)}{d(a,b;1)}
	\]
	give the circular barycentric coordinates of $ (\cos(x),\sin(x))^{T} $ relative to circular arc of length less than $ \pi $ with vertices $ a\neq b $ on the unit circle in $ \mathbb{R}^{2} $ with center at the origin, see \cite{alfeld}.
	\subsection{Quantum trigonometric Bernstein bases}\label{subsecquantrigbases}
	Here we are going to give the definition of quantum trigonometric Bernstein basis functions of degree $ n $ over arbitrary interval.
	\begin{definition}\label{defqtrigbases}
		Let $ d(a,b;q^{i})\neq 0 $ for $ i=0,1,\ldots ,n, $ then we define quantum trigonometric Bernstein basis functions of degree $ n $ over the interval $ [a,b] $ as follows:
		\begin{equation}\label{eqnqtrigbases}
			B^{n}_{k}(x;q)={n\brack k}_{q}\frac{\displaystyle \prod_{i=0}^{k-1}d(a,x;q^{i})\cdot \prod_{i=0}^{n-k-1}d(x,b;q^{i})}{\displaystyle \prod_{i=0}^{n-1}d(a,b;q^{i})},\quad  k=0,1,\ldots,n,
		\end{equation}
		where $ {n\brack k}_{q} $ is the $ q $-binomial coefficient defined by
		\[
		{n\brack k}_{q}=\frac{[n]_{q}!}{[k]_{q}![n-k]_{q}!}.
		\]
	\end{definition}
	Here
	\[
	[k]_{q}:=\left\{\begin{array}{cl}
		\frac{1-q^{k}}{1-q},&q\neq 1,\\
		k,&q=1
	\end{array}\right.
	\]
	is the $ q- $integer and 
	\[
	[k]_{q}!:=\left\{\begin{array}{cl}
		[k]_{q}[k-1]_{q}\ldots [1]_{q},&k=1,2,\ldots ,\\
		1,&k=0
	\end{array}\right.
	\]
	is the $ q $-factorial.
	
	Note that since $ d(x,y;1)=\sin(y-x), $ setting $ q=1 $ in Equation \eqref{eqnqtrigbases} gives the circular Bernstein-B\'ezier polynomials over arbitrary interval defined in \cite{alfeld}.
	
	The following figures show quantum trigonometric Bernstein basis functions over two intervals and for different values of $ q. $
	
	\begin{figure}[H]
		\centering
		\includegraphics[width=0.7\linewidth]{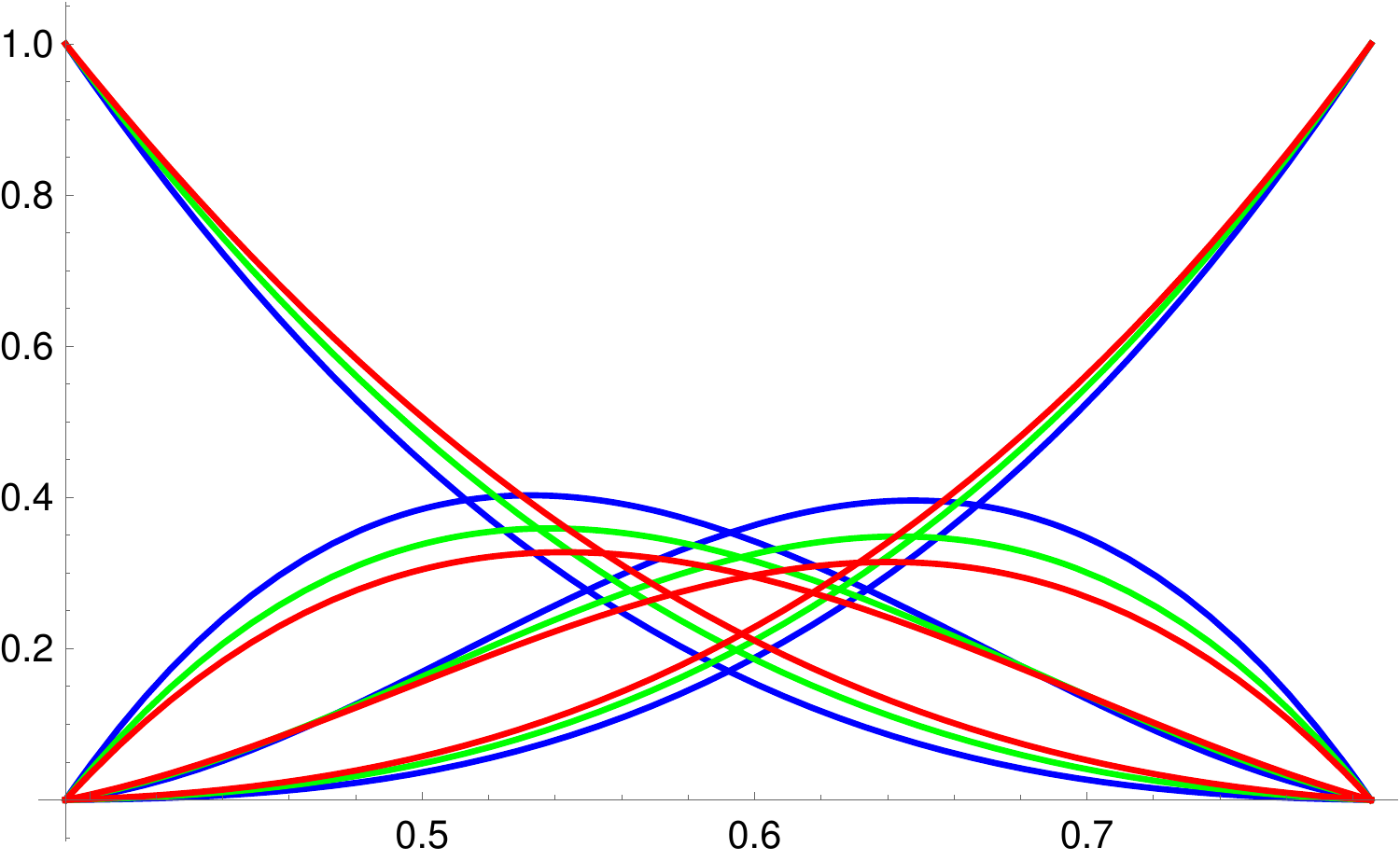}
		\caption{Cubic quantum trigonometric basis functions on $ [a,b]=[\pi/8,\pi/4] $ for $q=1.1$ (blue), $q=1.2$ (green) and $q=1.3$ (red).}
		\label{fig:basesq>1}
	\end{figure}
	
	\begin{figure}[H]
		\centering
		\includegraphics[width=0.7\linewidth]{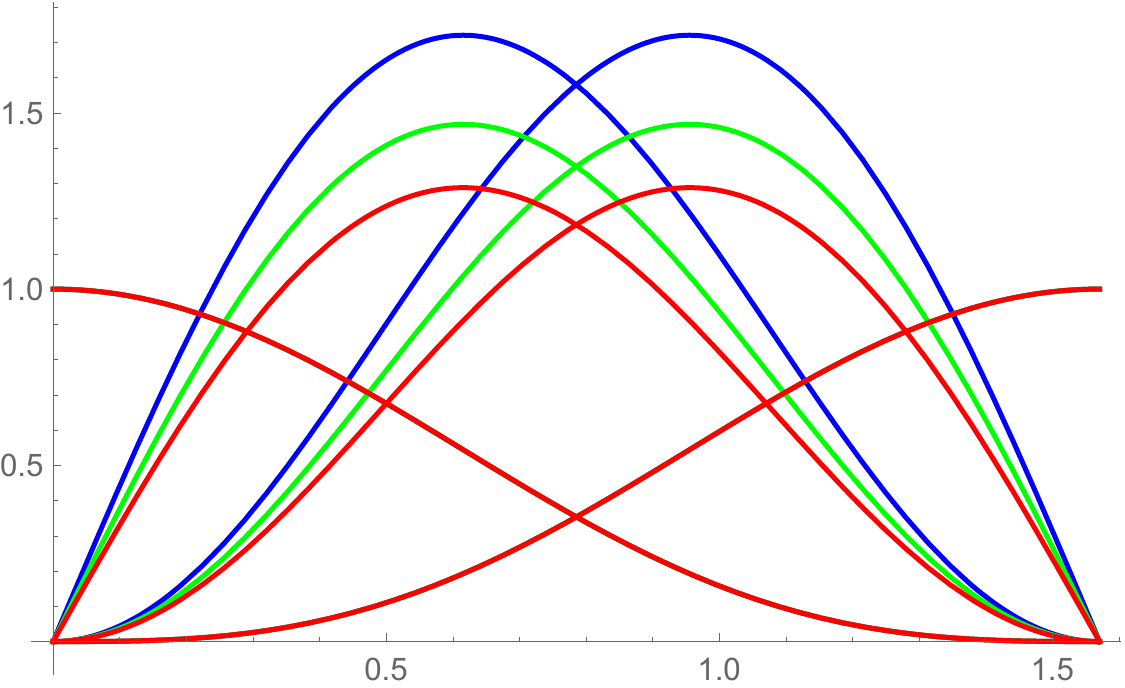}
		\caption{Cubic quantum trigonometric basis functions on $ [a,b]=[0,\pi/2] $ for $q=1.1$ (blue), $q=1.2$ (green) and $q=1.3$ (red).}
		\label{fig:0pibolu2}
	\end{figure}
	
	In Figure \ref{fig:0pibolu2} the basis function $ B^{3}_{0}(x;q)=\cos^{3}(x) $ does not depend on the parameter value $ q, $ so its graphs are overlap for different values of $ q. $ Similarly, the graphs of basis function $ B^{3}_{3}(x;q) $ are overlap since $ B^{3}_{3}(x;q)=\sin^{3}(x). $
	
	By using Pascal type identities of $ q $-binomial coefficient, it can be easily shown that quantum trigonometric Bernstein basis functions satisfy the following recurrence relations:
	\begin{align}\label{recrel1}
		B^{n}_{k}(x;q)=q^{n-k}\left(\frac{d(a,x;q^{k-1})}{d(a,b;q^{n-1})}\right)B^{n-1}_{k-1}(x;q)+\left(\frac{d(x,b;q^{n-k-1})}{d(a,b;q^{n-1})}\right)B^{n-1}_{k}(x;q)
	\end{align}
	and
	\begin{align}\label{recrel2}
		B^{n}_{k}(x;q)=\left(\frac{d(a,x;q^{k-1})}{d(a,b;q^{n-1})}\right)B^{n-1}_{k-1}(x;q)+q^{k}\left(\frac{d(x,b;q^{n-k-1})}{d(a,b;q^{n-1})}\right)B^{n-1}_{k}(x;q).
	\end{align}
	
	It is well known that totally positive bases present good shape preserving properties due to the variation diminishing properties of totally positive matrices (see \cite{carnicer}, \cite{goodman96}). The following subsection is about total positivity of quantum trigonometric Bernstein bases.
	
	\subsection{Total positivity}\label{sectotalpositivity}
	In Computer Aided Geometric Design, total positivity is a useful tool to study the shape preserving properties of curves. A sequence of real-valued functions $ \{\phi_0,\ldots,\phi_n\} $ is said to be totally positive on an interval $ I $ if the collocation matrix $ (\phi_i(x_j)), $ $ i=0,\ldots ,n, $ $ j=0,\ldots,m $ is totally positive for any points $ x_0<\ldots <x_m $ in $ I. $ A matrix is said totally positive if all its minors are nonnegative (see \cite{carnicer}).
	
	In \cite{goodman96}, total positivity of classical Bernstein B\'ezier bases is showed by using the following properties of totally positive functions:
	\begin{itemize}
		\item[P1:] If $ f $ is an increasing function from $ J $ into $ I, $ then $ \{\phi_0\circ f,\ldots,\phi_n\circ f\} $ is totally positive on $ J. $
		\item[P2:] If $ g $ is a positive function on $ I , $ then $ \{g\phi_0,\ldots,g\phi_n\} $ is totally positive on $ I. $
		\item[P3:] If $ A $ is a constant $ (m + 1) \times (n + 1) $ totally positive matrix and
		\[
		\psi_i=\sum_{j=0}^{n}A_{ij}\phi_j,\quad i=0,\ldots,m,
		\]
		then $ \{\psi_0,\ldots,\psi_n\} $ is totally positive on $ I. $
	\end{itemize}
	These properties can be easily derived from the definition. (See \cite{goodman96}). We will use these properties to show total positivity of the quantum trigonometric Bernstein basis functions.
	\begin{theorem}\label{thmtotalpositivity}
		If $ q> 0, $ then the basis $ \{B^{n}_{0}(x;q),B^{n}_{1}(x;q),\ldots ,B^{n}_{n}(x;q)\} $ is totally positive on  $ \left[\frac{k\pi}{2},\frac{(k+1)\pi}{2}\right],\ k\in \mathbb{Z}. $ 
	\end{theorem}
	\begin{proof}
		We have four cases, i.e. $ k\equiv i\ (\operatorname{mod}\ 4),$ $ i=0,1,2,3.$ We prove the theorem only for the case $ k\equiv 0\ (\operatorname{mod}\ 4). $ The cases $ k\equiv i\ (\operatorname{mod}\ 4),\ i=1,2,3 $ are omitted since their proofs are similar.
		
		Let $ k\equiv 0\ (\operatorname{mod}\ 4), $ then we have
		\[
		B^{n}_{i}(x;q)=q^{i^2-ni}{n \brack i}_{q}\left(d\left(\frac{k\pi}{2},x;1\right)\right)^{i}\left(d\left(x,\frac{(k+1)\pi}{2};1\right)\right)^{n-i}
		\]
		since
		\begin{align*}
			d\left(\frac{k\pi}{2},x;q^{i}\right)=q^{i}d\left(\frac{k\pi}{2},x;1\right)
		\end{align*}
		and
		\begin{align*}
			d\left(x,\frac{(k+1)\pi}{2};q^{i}\right)=q^{i}d\left(x,\frac{(k+1)\pi}{2};1\right).
		\end{align*}
		Clearly, the function $ d\left(\frac{k\pi}{2},x;1\right) $ is increasing and the function $ d\left(x,\frac{(k+1)\pi}{2};1\right) $ is decreasing on $ \left[\frac{k\pi}{2},\frac{(k+1)\pi}{2}\right]. $ On the other hand, both functions are positive on $ \left(\frac{k\pi}{2},\frac{(k+1)\pi}{2}\right). $ 
		
		It is given in \cite{goodman96} that the basis $ \{1,x,\ldots ,x^{n}\} $ is totally positive on $ [0,\infty). $ Then it follows by Property P1 with
		\[ f(x)=\frac{d\left(\frac{k\pi}{2},x;1\right)}{d\left(x,\frac{(k+1)\pi}{2};1\right)}
		\]
		that the basis
		\[
		\left\{ 1,\frac{d\left(\frac{k\pi}{2},x;1\right)}{d\left(x,\frac{(k+1)\pi}{2};1\right)},\ldots ,\frac{\left(d\left(\frac{k\pi}{2},x;1\right)\right)^{n}}{\left(d\left(x,\frac{(k+1)\pi}{2};1\right)\right)^{n}} \right\}
		\]
		is totally positive on $ \displaystyle \left(\frac{k\pi}{2},\frac{(k+1)\pi}{2}\right). $ Now, by Property P2 with $ g(x)=\left(d\left(x,\frac{(k+1)\pi}{2};1\right)\right)^{n}, $ we see that the basis
		\[
		\left\{ \left(d\left(x,\frac{(k+1)\pi}{2};1\right)\right)^{n},d\left(\frac{k\pi}{2},x;1\right)\left(d\left(x,\frac{(k+1)\pi}{2};1\right)\right)^{n-1},\ldots ,\left(d\left(\frac{k\pi}{2},x;1\right)\right)^{n} \right\}
		\]
		is totally positive on $ \displaystyle \left(\frac{k\pi}{2},\frac{(k+1)\pi}{2}\right). $ 
		
		Let $ A $ be a $ (n+1)\times (n+1) $ diagonal matrix with $ a_{ii}=q^{i^2-ni}{n \brack i}_{q}. $ The diagonal elements of $ A $ are positive for $ q>0, $ hence $ A $ is totally positive constant matrix.  Thus, since
		\[
		\left[\begin{array}{c}
			B^{n}_{0}(x;q)\\
			B^{n}_{1}(x;q)\\
			\vdots\\
			B^{n}_{n}(x;q)
		\end{array}\right]
		=A
		\left[\begin{array}{c}
			\left(d\left(x,\frac{(k+1)\pi}{2};1\right)\right)^{n}\\
			d\left(\frac{k\pi}{2},x;1\right)\left(d\left(x,\frac{(k+1)\pi}{2};1\right)\right)^{n-1}\\
			\vdots \\
			\left(d\left(\frac{k\pi}{2},x;1\right)\right)^{n}
		\end{array}\right],
		\]
		we see by Property P3 that the basis $ \{B^{n}_{0}(x;q),B^{n}_{1}(x;q),\ldots ,B^{n}_{n}(x;q)\} $ is totally positive on $ \displaystyle \left(\frac{k\pi}{2},\frac{(k+1)\pi}{2}\right) $ and the theorem follows by the continuity of the basis functions.
	\end{proof}

	We end this section by two important property of quantum trigonometric Bernstein bases:
	\begin{itemize}
		\item Endpoint property: 
		\begin{align}
			B^{n}_{i}(a;q)=\left\{
			\begin{array}{cc}
				1,&i=0\\
				0,&i\neq0.
			\end{array}
			\right.
		\end{align}
		and	
		\begin{align}
			B^{n}_{i}(b;q)=\left\{
			\begin{array}{cc}
				1,&i=n\\
				0,&i\neq n,
			\end{array}
			\right.
		\end{align}
		\item Non-negativity: If $ q> 0, $ then the basis $ \{B^{n}_{0}(x;q),B^{n}_{1}(x;q),\ldots ,B^{n}_{n}(x;q)\} $ is non-negative  on  $ \left[\frac{k\pi}{2},\frac{(k+1)\pi}{2}\right],\ k\in \mathbb{Z}. $
	\end{itemize}
	These properties are used in the following sections. The first property above follows easily from the definition and the latter is a consequence of total positivity of the basis.

	\section{Quantum trigonometric B\'ezier curves}\label{subsecquanbezcurve}
	Here we are going to give an explicit definition of quantum trigonometric B\'ezier curves. We also give two different recursive evaluation algorithm for quantum trigonometric B\'ezier curves.
	\begin{definition}\label{defqtrigbezcurve}
		A quantum trigonometric B\'ezier curve $ P $ over interval $ [a, b] $ is defined by
		\begin{align}\label{eqnqtrigbezcurve}
			\displaystyle P(x)=\sum_{k=0}^{n}\mathbf{b}_{k}B^{n}_{k}(x;q), \quad a\leqslant x\leqslant b.
		\end{align}	
	\end{definition}
	
	The coefficients $ \mathbf{b}_{k},\ k=0,1,\ldots n, $ are called control points of quantum trigonometric B\'ezier curve $ P. $ Note that the representation in Equation \eqref{eqnqtrigbezcurve} is an element from the space of trigonometric polynomials
	\[
	T_{n}:=\operatorname{span}\{\sin^{n-k}(x)\cos^{k}(x)\}_{k=0}^{n}.
	\]
	It is shown in \cite{lyche} that
	\[
	T_{n}=\left\{\begin{array}{ll}
		\operatorname{span}\{1,\cos(2x),\sin(2x),\ldots , \cos(nx),\sin(nx)\}, & n\text{ is even,}\\
		\operatorname{span}\{\cos(x),\sin(x),\cos(3x),\sin(3x),\ldots , \cos(nx),\sin(nx)\}, & n\text{ is odd.}
	\end{array}\right.
	\]
	
	Quantum trigonometric B\'ezier curves can be evaluated recursively as follows:
	\begin{theorem}\label{thmdecastalg}
		Let $ P $ be a quantum trigonometric B\'ezier curve over the interval $ [a,b] $ with the control points $ \mathbf{b}_{k},\ k=0,1,\ldots n. $ Set
		\[
		\tilde{b}^{0}_{k}(x)=\bar{b}^{0}_{k}(x)=\mathbf{b}_{k}
		\]
		$ k=0,1,\ldots,n $ and define recursively\\
		\textbf{Algorithm 1:}
		\[
		\tilde{b}^{r+1}_{k}(x)=q^k \frac{d(x,b;q^{n-r-k-1})}{d(a,b;q^{n-r-1})}\tilde{b}^{r}_{k}(x)+\frac{ d(a,x;q^{k})}{d(a,b;q^{n-r-1})}\tilde{b}^{r}_{k+1}(x)
		\]
		for $ r=0,1,\ldots,n-1, $ $ k=0,1,\ldots,n-r-1. $ Then
		\begin{align}\label{explicit1}
			\tilde{b}^{r}_{k}(x)=\sum_{j=0}^{r}q^{k (r-j)}\tilde{b}_{k+j}{r\brack j}\frac{\displaystyle \prod _{i=0}^{j-1} d(a,x;q^{i+k})\cdot \prod _{i=0}^{r-j-1} d(x,b;q^{i+n-r-k})}{\displaystyle \prod _{i=0}^{r-1} d(a,b;q^{i+n-r})}.	
		\end{align}
		\textbf{Algorithm 2:}
		\[
		\bar{b}^{r+1}_{k}(x)=\frac{ d(x,b;q^{n-r-k-1})}{d(a,b;q^{n-r-1})}\bar{b}^{r}_{k}(x)+q^{n-r-k-1}\frac{  d(a,x;q^{k})}{d(a,b;q^{n-r-1})}\bar{b}^{r}_{k+1}(x)
		\]
		for $ r=0,1,\ldots,n-1, $ $ k=0,1,\ldots,n-r-1. $ Then
		\begin{align}\label{explicit2}
			\bar{b}^{r}_{k}(x)=\sum_{j=0}^{r}q^{j(n-r-k)}\bar{b}_{k+j}{r\brack j}\frac{\displaystyle \prod _{i=0}^{j-1} d(a,x;q^{i+k})\cdot \prod _{i=0}^{r-j-1} d(x,b;q^{i+n-r-k})}{\displaystyle \prod _{i=0}^{r-1} d(a,b;q^{i+n-r})}.	
		\end{align}
		In particular $ \tilde{b}^{n}_{0}(x)=\bar{b}^{n}_{0}(x)=P(x). $
	\end{theorem}
	\begin{proof} We will only prove equation \eqref{explicit1}, since the proof of equation \eqref{explicit2} is similar.
		
		The proof is by induction on $ r. $ Clearly, \eqref{explicit1} is true for $ r=0,\ k=0,\ldots,n-1. $ Assume that \eqref{explicit1} is true for any $ r<n,\ k=0,\ldots,n-r-1, $ then
		\begin{align*}
			\tilde{b}^{r+1}_{k}(x)=&q^k \frac{d(x,b;q^{n-r-k-1})}{d(a,b;q^{n-r-1})}\tilde{b}^{r}_{k}(x)+\frac{ d(a,x;q^{k})}{d(a,b;q^{n-r-1})}\tilde{b}^{r}_{k+1}(x)\\
			=&q^k \frac{d(x,b;q^{n-r-k-1})}{d(a,b;q^{n-r-1})}\sum_{j=0}^{r}q^{k (r-j)}\tilde{b}_{k+j}{r\brack j}\frac{\displaystyle \prod _{i=0}^{j-1} d(a,x;q^{i+k})\cdot \prod _{i=0}^{r-j-1} d(x,b;q^{i+n-r-k})}{\displaystyle \prod _{i=0}^{r-1} d(a,b;q^{i+n-r})}\\
			&+\frac{ d(a,x;q^{k})}{d(a,b;q^{n-r-1})}\sum_{j=0}^{r}q^{(k+1) (r-j)}\tilde{b}_{k+1+j}{r\brack j}\frac{\displaystyle \prod _{i=0}^{j-1} d(a,x;q^{i+k+1})\cdot \prod _{i=0}^{r-j-1} d(x,b;q^{i+n-(r+1)-k})}{\displaystyle \prod _{i=0}^{r-1} d(a,b;q^{i+n-r})}.
		\end{align*}
		Rearranging gives
		\begin{align*}
			\tilde{b}^{r+1}_{k}(x)=\sum_{j=0}^{r}q^{k ((r+1)-j)}\tilde{b}_{k+j}&{r\brack j}\frac{\displaystyle \prod _{i=0}^{j-1} d(a,x;q^{i+k})\cdot \prod _{i=-1}^{r-j-1} d(x,b;q^{i+n-r-k})}{\displaystyle \prod _{i=-1}^{r-1} d(a,b;q^{i+n-r})}\\
			&+\sum_{j=0}^{r}q^{k (r-j)}\tilde{b}_{k+1+j}q^{r-j}{r\brack j}\frac{\displaystyle \prod _{i=-1}^{j-1} d(a,x;q^{i+k+1})\cdot \prod _{i=0}^{r-j-1} d(x,b;q^{i+n-(r+1)-k})}{\displaystyle \prod _{i=-1}^{r-1} d(a,b;q^{i+n-r})}.
		\end{align*}
		Shifting the index of the products, we get
		\begin{align*}
			\tilde{b}^{r+1}_{k}(x)=\sum_{j=0}^{r}q^{k ((r+1)-j)}\tilde{b}_{k+j}&{r\brack j}\frac{\displaystyle \prod _{i=0}^{j-1} d(a,x;q^{i+k})\cdot \prod _{i=0}^{(r+1)-j-1} d(x,b;q^{i+n-(r+1)-k})}{\displaystyle \prod _{i=0}^{(r+1)-1} d(a,b;q^{i+n-(r+1)})}\\
			&+\sum_{j=0}^{r}q^{k (r-j)}\tilde{b}_{k+1+j}q^{r-j}{r\brack j}\frac{\displaystyle \prod _{i=}^{j} d(a,x;q^{i+k})\cdot \prod _{i=0}^{r-j-1} d(x,b;q^{i+n-(r+1)-k})}{\displaystyle \prod _{i=0}^{(r+1)-1} d(a,b;q^{i+n-(r+1)})}.
		\end{align*}
		Now, shifting the index of the second summation and rearranging gives
		\begin{align*}
			\tilde{b}^{r+1}_{k}(x)=q^{k (r+1)}\tilde{b}_{k}&\frac{\displaystyle \prod _{i=0}^{(r+1)-1} d(x,b;q^{i+n-(r+1)-k})}{\displaystyle \prod _{i=0}^{(r+1)-1} d(a,b;q^{i+n-(r+1)})}\\
			&+\sum_{j=1}^{r}q^{k ((r+1)-j)}\tilde{b}_{k+j}\left\{{r\brack j}+q^{(r+1)-j}{r\brack j-1}\right\} \frac{\displaystyle \prod _{i=0}^{j-1} d(a,x;q^{i+k})\cdot \prod _{i=0}^{(r+1)-j-1} d(x,b;q^{i+n-(r+1)-k})}{\displaystyle \prod _{i=0}^{(r+1)-1} d(a,b;q^{i+n-(r+1)})}\\
			&+\tilde{b}_{k+r+1}\frac{\displaystyle \prod _{i=0}^{(r+1)-1} d(a,x;q^{i+k})}{\displaystyle \prod _{i=0}^{(r+1)-1} d(a,b;q^{i+n-(r+1)})}.
		\end{align*}
		Hence, equation \eqref{explicit1} follows from Pascal type identity $ \displaystyle {r\brack j}+q^{r+1-j}{r\brack j-1}={r+1\brack j}. $ 
	\end{proof}
	Using the endpoint property of the basis functions, one may easily deduce that	$ P(a)=\mathbf{b}_{0} $ and $ P(b)=\mathbf{b}_{n}. $ This property of B\'ezier curves is called the end point interpolation property. Another nice shape preserving property of quantum trigonometric B\'ezier curves follows from total positivity of the basis functions:
	
	As in \cite{goodman96}, for any real sequence $ v, $ finite or infinite, we use the notation $ S^{-}(v) $ for the number of strict sign changes in $ v. $ Also, the strict sign changes of a real valued function $ f $ on an interval $ I, $ denoted by $ S^{-}(f) $ is defined as
	\[
	S^{-}(f)=\sup S^{-}(f(x_{0}),\ldots, f(x_{m})),
	\]
	where the supremum is taken over all increasing sequences  $ (x_{0},\ldots, x_{m}). $
	
	As a consequence of Theorem \ref{thmtotalpositivity}, we have the following corollary.
	\begin{corollary}\label{corollarysignchanges}
		Let $ q>0 $ and $ P $ be a quantum trigonometric B\'ezier curve over the interval $ \left[\frac{k\pi}{2},\frac{(k+1)\pi}{2}\right],\ k\in \mathbb{Z}, $  with the control points $ \mathbf{b}_{k},\ k=0,1,\ldots n. $ Then,
		\[
		S^{-}(P)\leqslant S^{-}(\mathbf{b}_{0},\ldots ,\mathbf{b}_{n}).
		\]
	\end{corollary}
	
	It is seen from Corollary \ref{corollarysignchanges} that the number of sign changes of a quantum trigonometric B\'ezier curve defined over the interval $ \left[\frac{k\pi}{2},\frac{(k+1)\pi}{2}\right],\ k\in \mathbb{Z}, $ is restricted with the number of the sign changes of its control polygon. We have also seen that these curves satisfy the end point interpolation property. So one may say that, in some sense, the quantum trigonometric B\'ezier curves defined over the interval $ \left[\frac{k\pi}{2},\frac{(k+1)\pi}{2}\right],\ k\in \mathbb{Z} $ mimics the shape of its control polygon. However these shape preserving properties do not include the convex hull and affine invariance properties. This disadvantage of quantum trigonometric B\'ezier curves is due to the quantum Bernstein basis functions' lack of the partition of unity property. For better shape preserving properties, one may consider the rational counterpart of the quantum trigonometric B\'ezier curves which we are going to discuss in Section \ref{secrational}.	We end this section with the following discussion on a well known property of B\'ezier curves:
		
		Let $P(x)$ be a B\'ezier curve defined over the interval $[a,b].$ For any $c\in [a,b],$ let $P_{l}(x)$ and $P_{r}(x)$ be the curve segments of $P$ that corresponds to the intervals $[a,c]$ and $[c,d]$ respectively. These two curve segments can be expressed as a B\'ezier curve of degree $n$ defined over the interval $[a,b].$ This porperty of B\'ezier curves is called subdivision property (see \cite{farin}). In the classical case, i.e. standard polynomial B\'ezier curves, the control points of the curve segments $P_{l}$ and $P_{r}$ can be easily obtained from the classical de Casteljau algorithm. However, for the $q$-B\'ezier curves, the control points of each curve segment are obtained from two different de Casteljau type algortihm (see \cite{disibuyuk22}, \cite{simeonov12}). For the quantum trigonometric B\'ezier curves, both \textbf{\textit{Algorithm 1}} and \textbf{\textit{Algorithm 2}} do not lead us to the control points of the curve segments. Considering the importance of the subdivision property, our future work are planned to find the control points of the curve segments by applying techniques similar to that presented in \cite{disibuyuk15}, \cite{disibuyuk22} and \cite{simeonov12}.

	\section{Rational quantum trigonometric B\'ezier curves}\label{secrational}
	We begin with the definition of rational quantum trigonometric Bernstein bases.
	\begin{definition}\label{defrationalbases}
		For a given $ w_{k}\in\mathbb{R},\ k=0,1,\ldots ,n, $ the rational quantum trigonometric Bernstein bases $ R^{n}_{k}(x) $ of degree $ n $ over the interval $ [a,b] $ are defined as follows
		\begin{align}\label{eqnrationalbases}
			R^{n}_{k}(x;q)=\frac{w_{k}B^{n}_{k}(x;q)}{\sum_{i=0}^{n}w_{i}B^{n}_{i}(x;q)}.
		\end{align}
	\end{definition}
	Here the numbers $ w_{k} $ are called weights. For non-singularity, the weights must be chosen so that
	\[
	\sum_{i=0}^{n}w_{i}B^{n}_{i}(x;q)\neq 0, \forall x\in[a,b].
	\]
	The following figure depicts the third degree rational quantum trigonometric Bernstein bases defined over the interval $ \left[0,\frac{\pi}{2}\right]. $
	\begin{figure}[H]
		\centering
		\includegraphics[width=0.7\linewidth]{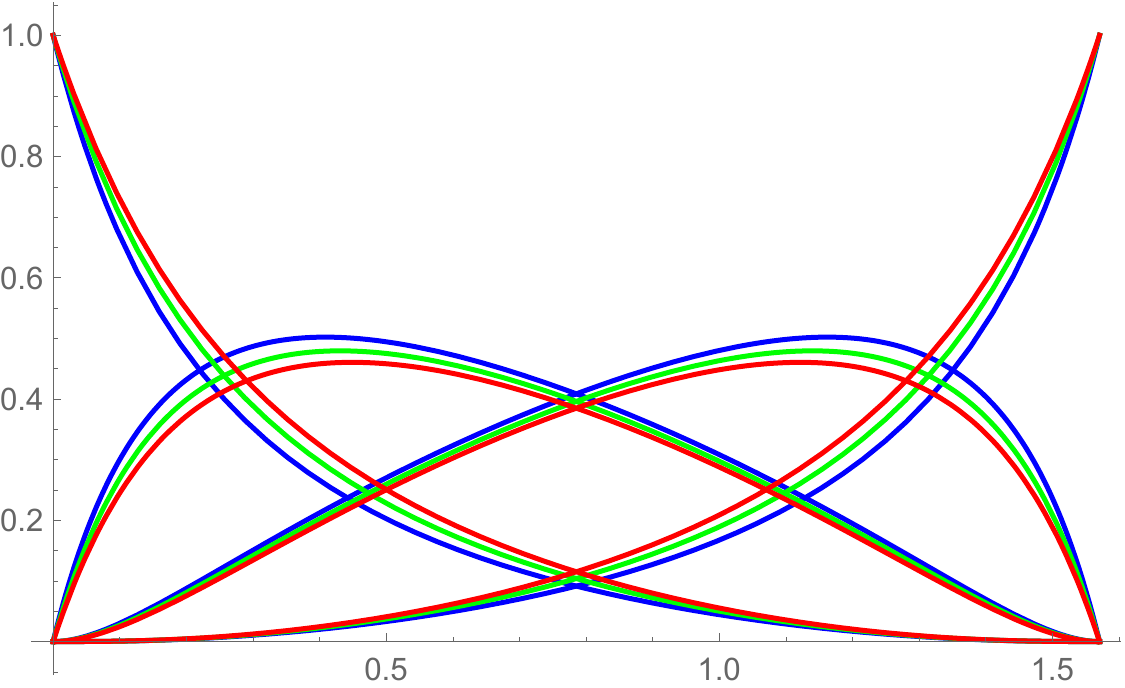}
		\caption{Third degree rational quantum trigonometric Bernstein bases, with $q=1.1$ (blue), $q=1.2$ (green) and $q=1.3$ (red), defined over the interval $ \left[0,\frac{\pi}{2}\right] $ where all weights are equal to 1.}
		\label{fig:0pibolu2normalized}
	\end{figure}
	\begin{definition}\label{defrationalcurve}
		A rational quantum trigonometric B\'ezier curve $ R $ over interval $ [a, b] $ is defined by
		\begin{align}\label{eqnrationalcurve}
			\displaystyle R(x)=\sum_{k=0}^{n}\mathbf{b}_{k}R^{n}_{k}(x;q)=\sum_{k=0}^{n}\mathbf{b}_{k}\frac{w_{k}B^{n}_{k}(x;q)}{\sum_{i=0}^{n}w_{i}B^{n}_{i}(x;q)}, \quad a\leqslant x\leqslant b.
		\end{align}	
	\end{definition}
	The following theorem will be useful for the shape preserving properties of rational quantum trigonometric B\'ezier curves.
	\begin{theorem}\label{thmtotalpositivityofrationalbases}
		If $ q> 0 $ and $ w_{k}> 0,\ k=0,1,\ldots n, $ then the basis $ \{R^{n}_{0}(x;q),R^{n}_{1}(x;q),\ldots ,R^{n}_{n}(x;q)\} $ is normalized totally positive on  $ \left[\frac{k\pi}{2},\frac{(k+1)\pi}{2}\right],\ k\in \mathbb{Z}. $
	\end{theorem}
	\begin{proof}
		We may write
		\[
		\left[\begin{array}{c}
			R^{n}_{0}(x;q)\\
			R^{n}_{1}(x;q)\\
			\vdots\\
			R^{n}_{n}(x;q)
		\end{array}\right]=W\left[\begin{array}{c}
			\tilde{B}^{n}_{0}(x;q)\\
			\tilde{B}^{n}_{1}(x;q)\\
			\vdots\\
			\tilde{B}^{n}_{n}(x;q)
		\end{array}\right]
		\]
		where
		\[
		\{\tilde{B}^{n}_{0}(x;q),\ldots ,\tilde{B}^{n}_{n}(x;q)\}=\left\{\frac{B^{n}_{0}(x;q)}{\sum_{i=0}^{n}w_{i}B^{n}_{i}(x;q)},\ldots ,\frac{B^{n}_{n}(x;q)}{\sum_{i=0}^{n}w_{i}B^{n}_{i}(x;q)}\right\}
		\]
		and
		\[
		W=\left[
		\begin{array}{ccc}
			w_{0} & & \\
			& \ddots & \\
			& & w_{n}
		\end{array}
		\right]
		\]
		is an $ (n+1)\times (n+1)$ diagonal matrix.
		
		If $ w_{k}> 0,\ k=0,\ldots , n, $ then the function $ \displaystyle \frac{1}{\sum_{i=0}^{n}w_{i}B^{n}_{i}(x;q)} $ is positive. Hence the basis $ \{\tilde{B}^{n}_{0}(x;q),\ldots ,\tilde{B}^{n}_{n}(x;q)\} $ is totally positive on  $ \left[\frac{k\pi}{2},\frac{(k+1)\pi}{2}\right],\ k\in \mathbb{Z} $ by Property P2 given in Section \ref{sectotalpositivity}. Thus totally positivity of the basis $ \{R^{n}_{0}(x;q),\ldots ,R^{n}_{n}(x;q) \} $ follows from Property P3. Moreover, the basis $ \{R^{n}_{0}(x;q),\ldots ,R^{n}_{n}(x;q) $ is normalized totally positive on  $ \left[\frac{k\pi}{2},\frac{(k+1)\pi}{2}\right],\ k\in \mathbb{Z}, $ since $ \sum_{k=0}^{n}R^{n}_{k}(x;q)=1. $
	\end{proof}
	\subsection{Shape preserving properties of rational quantum trigonometric B\'ezier curves}\label{subsecshapepreservingpropsofrationalcurves}
	It follows from Theorem \ref{thmtotalpositivityofrationalbases} that rational quantum trigonometric B\'ezier curves posses nice shape preserving properties. This properties are as follows:
	
	The  rational quantum trigonometric B\'ezier curve $ R(x)=\sum_{k=0}^{n}\mathbf{b}_{k}R^{n}_{k}(x;q) $ over the interval $ [a,b]=\left[\frac{k\pi}{2},\frac{(k+1)\pi}{2}\right],\ k\in \mathbb{Z} $ satisfies
	\begin{itemize}
		\item \textbf{End point interpolation property:} 
		We have $ R(a)=\mathbf{b}_{0} $ and $ R(b)=\mathbf{b}_{n}. $
		\item \textbf{Convex hull property:} $ R $ lies in the convex hull of the control points $ \mathbf{b}_{k},\ k=0,1\ldots ,n $ when $ q>0 $ and all $ w_{k}>0. $
		\item \textbf{Variation diminishing property:} The curve has no more intersections with any line than its control polygon has.
		\item \textbf{Affine invariance property:} The curve $ R $ is an affine combination of its control points, hence it is invariant under affine transformations.
	\end{itemize}
	
	The following figure shows a control polygon and associated rational quantum trigonometric B\'ezier curves with different values of $ q. $
	\begin{figure}[H]
		\centering
		\includegraphics[width=0.4\linewidth]{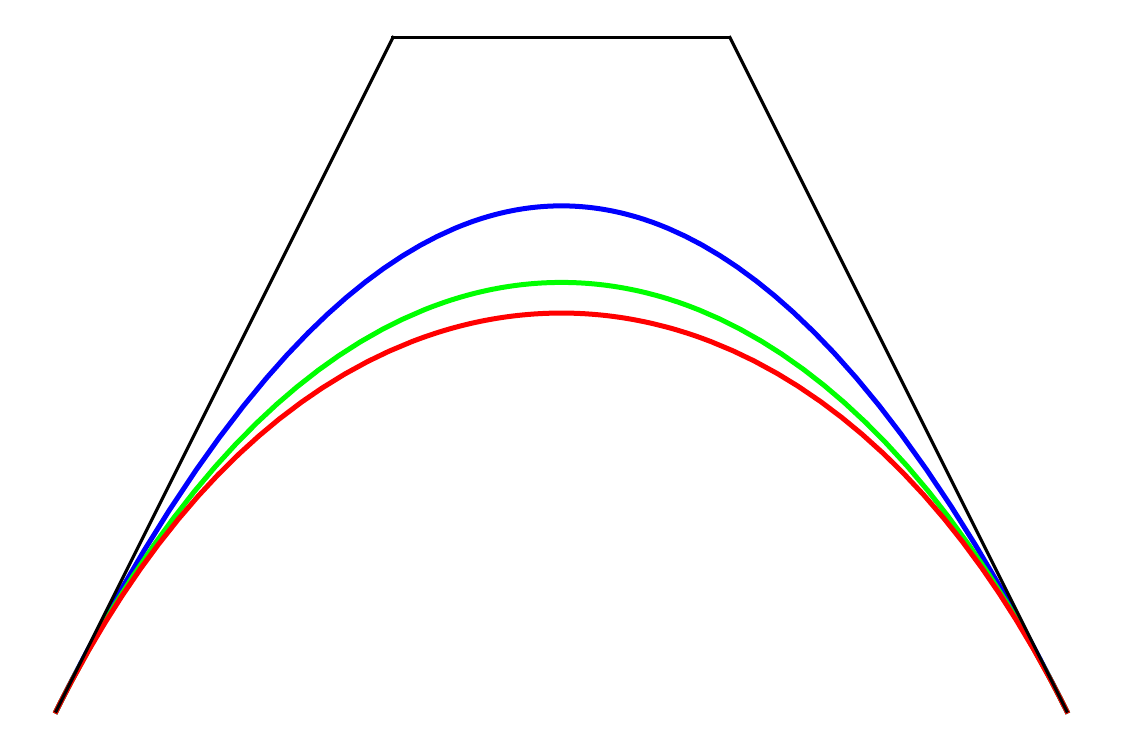}
		\caption{Third degree rational quantum trigonometric B\'ezier curves associated with the control polygon. Here $ \mathbf{b}_{0}=(0,0), $ $ \mathbf{b}_{1}=(1,2), $ $ \mathbf{b}_{2}=(2,2), $ $ \mathbf{b}_{3}=(3,0), $ all $ w_{k}=1 $ and $q=1$ (blue), $q=2$ (green) and $q=3$ (red). The rational quantum trigonometric B\'ezier curves are evaluated on the interval $ \left[0,\frac{\pi}{2}\right]. $}
		\label{fig:rationalcurve}
	\end{figure}
	Figure \ref{fig:rationalcurve} shows that as the parameter value $ q $ increases the curve approach the line segment between the points $ \mathbf{b}_{0} $ and $ \mathbf{b}_{3}. $


\end{document}